\documentclass[12pt]{article}

\usepackage{amsmath,amsfonts,amssymb, amsthm}
\usepackage[all]{xy}

\newcommand{\calC}{\ensuremath{\mathcal C}}

\newcommand{\calE}{\mathcal{E}} 
\newcommand{\calF}{\mathcal{F}}

\newcommand{\calS}{\ensuremath{\mathcal S}}

\newcommand{\Un}{\mathrm{Dec}}

\theoremstyle{plain}
\newtheorem{theorem}{Theorem}[section]
\newtheorem{corollary}[theorem]{Corollary}
\newtheorem{proposition}[theorem]{Proposition}
\newtheorem{lemma}[theorem]{Lemma}

\theoremstyle{definition}

\newtheorem{remark}[theorem]{Remark}

\begin{document}

\author{M. Menni\footnote{Conicet and Universidad Nacional de La Plata, Argentina.}}



\title{The Unity and Identity of \\ decidable objects and double-negation sheaves}

\maketitle

\begin{abstract}
Let $\calE$ be a topos,  ${\Un(\calE) \rightarrow \calE}$ be the full subcategory of decidable objects, and ${\calE_{\neg\neg}\rightarrow \calE}$ be the full subcategory of double-negation sheaves. We give sufficient conditions for the existence of a Unity and Identity ${\calE \rightarrow \calS}$ for the two subcategories of $\calE$ above, making them Adjointly Opposite. Typical examples of such $\calE$ include many `gros' toposes in Algebraic Geometry, simplicial sets and other toposes of `combinatorial' spaces in Algebraic Topology, and certain models of Synthetic Differential Geometry.
\end{abstract}


\section{Introduction}

The preface {\em Contributors to Sets for Mathematics} in \cite{LawvereRosebrugh} starts by stating that the ``book began as the transcript of a 1985 course at SUNY Buffalo and still retains traces of that verbal record". This is the same course that Lawvere mentions in the beginning of \cite{Lawvere94}; notice the acknowledgments to Myhill for making the course possible and for his incisive comment about Cantor's {\em lauter Einsen}.
Some of the ideas in the 1985 course were also published in \cite{Lawvere03} whose influence may be seen in more recent work on Axiomatic Cohesion \cite{LawvereMenni2015}. 
The present paper is another step in the line of work indicated above, providing further evidence of the soundness of Lawvere's interpretation of Cantor's work.

Let us recall from \cite{Lawvere96a} that 
``A {\em Unity and Identity} (UI) of two maps with a common codomain $C$ is a third map with domain $C$ which composes with both to give isomorphisms. The existence of such a third map obviously implies that the two maps are subobject-inclusions, that these two inclusions have isomorphic domains, and that $C$ is retracted onto both of these subobjects, but moreover that there is a common retraction in the following sense: any UI in any category is canonically isomorphic to one in which both composite isomorphisms are actually the identity map. In this view, a UI is just a map equipped with two sections, or equivalently, is a common retraction for two subobjects whose underlying objects are identical".
We recall also from that paper that: ``In a 2-category, two parallel maps may be called {\em adjointly opposite} (AO) if there is a single third map which is right adjoint to one of the given pair and left adjoint to the other. [...] In case the third map is {\em also} a UI for the adjointly opposites, then the AO are of course both full inclusions. Such a map, having both left and right adjoints which are moreover full inclusions, is a UIAO (unity and identity of
adjointly opposites), also known as an `essential localization'".

In \cite{Lawvere94}, published in 1994, Lawvere suggests that ``when studying the works of the great mathematicians of the last century we must strive afresh to find the core content of their thought, without being prejudiced by the opinions of the editors of their collected works, and others during the period after the last decade of the century". Moreover, he puts this in practice in an analysis of the work of Cantor which clarifies the distinction between {\em Mengen} and {\em Kardinalen}. In particular, he proposes to study  the ``general situation in which we are given an arbitrary category $\mathbf{M}$ of {\em Mengen}, itself containing two opposed subcategories of discrete and codiscrete objects, each essentially identical with a category $\mathbf{K}$ of {\em Kardinalen}". In other words, he proposes to study UIAOs ${\mathbf{M} \rightarrow \mathbf{K}}$ following the intuition that $\mathbf{M}$ is a `category of spaces',  that $\mathbf{K}$ is a category of `abstract sets' and that the UI ${\mathbf{M} \rightarrow \mathbf{K}}$ assigns, to the each space, the corresponding set of points. Also intuitively, the left and right adjoints ${\mathbf{K} \rightarrow \mathbf{M}}$ send a set to the associated discrete and codiscrete space respectively.

In \cite{Lawvere03}, it is observed that ``often the needed categories of spaces are self-founded in the sense that within them a subcategory playing the role of abstract structureless sets can be defined [...]". 
One general explanation of this observation is obtained in \cite{LawvereMenni2015} where it is proved that, for any cohesive geometric morphism ${p : \calE \rightarrow \calS}$ satisfying Stable Connected Codiscreteness, the subtopos ${p_*  \dashv p^! : \calS \rightarrow \calE}$ coincides with the subtopos ${\calE_{\neg\neg} \rightarrow \calE}$ of sheaves for the $\neg\neg$-topology in $\calE$. In other words, for such $p$, the codiscrete spaces may be defined as the sheaves for the $\neg\neg$-topology. As mentioned after Corollary~{9.4} loc.~cit., one  interpretation of that result involves the conclusion that the real contrast between {\em Mengen} and {\em Kardinalen} emerges from the case of a topos whose double-negation part has additional remarkable properties; in particular, the property that the subtopos of $\neg\neg$-sheaves is essential.
It is natural to wonder if, at least in the same context of a cohesive topos ${p : \calE \rightarrow \calS}$, the hyperconnected ${p^* \dashv p_* : \calE \rightarrow \calS}$ can also be defined without reference to $p$.

\begin{remark} Let us emphasize that ``the needed categories of spaces" mentioned above include many categories in standard mathematical practice, such as `gros' toposes in algebraic geometry \cite{MenniExercise}, simplicial sets and other toposes of `combinatorial' spaces \cite{MarmolejoMenni2017}, certain well-adapted models of synthetic differential geometry \cite{KockSDG2ed}, as well as the cohesive Grothendieck toposes introduced in \cite{Menni2014a}. It is not a coincidence that these are models of the axioms for Cohesion introduced in \cite{Lawvere07}.
\end{remark}

Recall that an object $X$ in an extensive category $\calE$ with finite limits is {\em decidable} if the diagonal ${\Delta : X \rightarrow X\times X}$ is complemented. For example, in a Boolean topos, every object is decidable. The full subcategory of decidable objects will be denoted by ${\Un(\calE) \rightarrow \calE}$.
The category ${\Un(\calE)}$ is extensive and has finite limits. The inclusion ${\Un(\calE) \rightarrow \calE}$ preserves finite coproducts and finite limits and it is closed under subobjects. (See, for example, \cite{Carboni96}.)

Also motivated by the work of Cantor, McLarty proposes in \cite{McLarty87} (see also \cite{McLarty88}) to consider 2-valued toposes $\calE$ with global support and such that every object $X$ in $\calE$ has a unique subobject ${\beta_X : C X \rightarrow  X}$ with ${C X}$ decidable and every global element of $X$ factoring through ${\beta_X}$. It follows easily that, under these hypotheses, the inclusion ${\Un(\calE) \rightarrow \calE}$ has a right adjoint but, moreover,  the right adjoint ${\calE \rightarrow \Un(\calE)}$ is a UIAO for  ${\Un(\calE) \rightarrow \calE}$ and ${\calE_{\neg\neg} \rightarrow \calE}$.

Our purpose is to show that the conclusion of McLarty's result holds for any stably pre-cohesive topos ${p : \calE \rightarrow \calS}$  (in the sense of \cite{LawvereMenni2015}) with Boolean codomain.  In more detail, let $\calE$ be a topos and consider the following statements.

\begin{enumerate}
\item The topos $\calS$ is Boolean and  ${p : \calE \rightarrow \calS}$ is a stably pre-cohesive geometric morphism. 
\item The subcategories  ${\Un(\calE) \rightarrow \calE}$ and ${\calE_{\neg\neg} \rightarrow \calE}$ are the left and right inclusions of a UIAO.
\item The inclusion ${\Un(\calE) \rightarrow \calE}$ has a right adjoint that reflects initial object.
\item The inclusion ${\calE_{\neg\neg} \rightarrow \calE}$ of $\neg\neg$ sheaves is the right inclusion of a UIAO.
\end{enumerate}

By Proposition~{4.4} in \cite{LawvereMenni2015}, the first item implies that the subtopos ${p_* \dashv p^! : \calS \rightarrow \calE}$ coincides with ${\calE_{\neg\neg} \rightarrow \calE}$ and so, that item implies the fourth.
Trivially, the second item also implies the fourth.
We prove here that the second and third items are equivalent, and that they are both implied by the first.

\begin{remark} It seems worth observing that the consideration of decidable objects is also applicable to sites. Indeed, several of the examples of sites $\calC$ for (pre-)cohesive toposes have the feature that the subcategory ${\Un(\calC) \rightarrow \calC}$ has additional remarkable properties. In particular, it has a finite-product preserving left adjoint. For example, finite posets and (the opposite of) $K$-algebras for nice fields $K$ \cite{Lawvere08, MenniExercise}.
\end{remark}

Assume now that $\calE$ and $\calS$ are toposes and let ${p : \calE \rightarrow \calS}$ be a geometric morphism with unit $\alpha$ and counit $\beta$.
Recall that $p$ is said to be {\em connected} if ${p^* : \calS \rightarrow \calE}$ is full and faithful. Recall also that the geometric morphism $p$ is {\em hyperconnected} if and only if it is connected and the counit $\beta$ is monic.

It is well-known that each object $B$ in $\calS$ determines a geometric morphism ${p/B : \calE/p^* B \rightarrow \calS/B}$ whose inverse image ${(p/B)^*}$ is simply $p^*$ applied to morphisms in $\calS$ with codomain $B$. See Example~{A4.1.3} in \cite{elephant}.
If  ${p^*}$ is fully faithful so is ${(p/B)^*}$ and, in this case,  ${(p/B)_* : \calE/p^* B \rightarrow \calS/B}$ sends ${x : X \rightarrow p^* B}$ to the composite
$$\xymatrix{
p_* X \ar[r]^-{p_* x} & p_* (p^* B) \ar[r]^-{\alpha^{-1}} & B
}$$
as an object in ${\calS/B}$. Notice that $\alpha$ is an isomorphism because $p$ is connected. In other words, if ${p : \calE \rightarrow \calS}$ is connected then so is ${p/B : \calE/p^* B \rightarrow \calS/B}$. See Lemma~{5.1} in \cite{LawvereMenni2015}.

\begin{lemma}\label{LemHyperconnectedAreStable} If the geometric  ${p : \calE \rightarrow \calS}$ is hyperconnected then so is  ${p/B : \calE/p^* B \rightarrow \calS/B}$ for every object $B$ in $\calS$.
\end{lemma}
\begin{proof}
By hypothesis, $p^*$ is fully faithful and the counit $\beta$ of ${p^* \dashv p_*}$ is monic. We already know that ${(p/B)^*}$ is fully faithful so we need only check that the counit of $p/B$ is monic.
Using the description above, it is easy to check that  the counit of ${(p/B)^* \dashv (p/B)_*}$ is the top map in the following diagram
$$\xymatrix{
p^* (p_* X) \ar[d]_-{p^*(p_* x)} \ar[rr]^-{\beta} & & X \ar[d]^-x \\ 
p^*(p_* (p^* B)) \ar[rr]_-{p^*(\alpha^{-1})} & & p^* B
}$$
thought of as a map  ${(p/B)^* ((p/B)_* x) \rightarrow x}$ in ${\calE/ p^* B}$.
\end{proof}

 If the map ${p : \calE \rightarrow \calS}$ is hyperconnected, the inclusions ${p^* : \calS \rightarrow \calE}$  and ${\Un(\calE) \rightarrow \calE}$ share the following important properties: preservation of finite limits and finite coproducts, and closure under subobjects.

\begin{lemma}\label{LemHyperConnectedIffQuotients} If the inclusion ${\Un(\calE) \rightarrow \calE}$ has a right adjoint then $\Un(\calE)$ is a topos and the coreflection is the direct image of a hyperconnected geometric morphism ${\calE \rightarrow \Un(\calE)}$. 
\end{lemma}
\begin{proof}
Let ${p^* : \Un(\calE) \rightarrow \calE}$ be the full subcategory of decidable objects.
If $p^*$ has a right adjoint  ${p_* : \calE \rightarrow \Un(\calE)}$  then $\Un(\calE)$ is the category  of coalgebras for a lex (and idempotent) comonad and hence it is a topos.
Also, $p_*$  is the direct image of a connected  geometric morphism ${p : \calE \rightarrow \Un(\calE)}$ and, since ${p^* : \Un(\calE) \rightarrow \calE}$ is closed under subobjects, $p$ is hyperconnected by Proposition~{A4.6.6} in \cite{elephant}.
\end{proof}

The following is surely a folklore result about extensive categories with finite products. We state it for toposes because we are dealing mainly with these.

\begin{lemma}\label{LemExtensiveFolk} Let $\calE$ and $\calS$ be toposes, and let  ${F : \calE \rightarrow \calS}$ be a functor that preserves finite products and finite coproducts,  and also reflects initial object. Then, for any $X$ in $\calE$, if $X$ is decidable and ${F X}$ is subterminal then $X$ is subterminal.
\end{lemma}
\begin{proof}
Since $X$ is decidable we have a coproduct diagram as on the left below
$$\xymatrix{
X \ar[r]^-{\Delta} & X \times X & \ar[l]_-k  K & F X \ar[r]^-{\Delta} & F X \times F X & \ar[l]_-{F k} F K
}$$
and, since ${F : \calE \rightarrow \calS}$ preserves finite products and finite coproducts, the cospan on the right above is also a coproduct diagram. Since ${F X}$ is subterminal, ${\Delta  : F X \rightarrow F X \times F X}$ is an isomorphism, so ${F K}$ is initial. As $F$ reflects initial object by hypothesis, $K$ is initial, and then ${\Delta : X \rightarrow X\times X}$ is an isomorphism; which means that $X$ is subterminal.
\end{proof}

\section{Decidable objects in the domain of an essential map}

Let ${p : \calE \rightarrow \calS}$ be a geometric morphism.
It is called {\em essential} if $p^*$ has a left adjoint, typically denoted by ${p_! : \calE \rightarrow \calS}$.
For brevity and emphasis, we introduce the following ad-hoc terminology. The morphism $p$ is called {\em pressential} if it is essential and the left adjoint ${p_! : \calE \rightarrow \calS}$ preserves finite products.

\begin{lemma}\label{LemDecidableSubconnectedImpliesSubterminal} Assume that ${p : \calE \rightarrow \calS}$ is pressential. If $X$ in $\calE$ is decidable and ${p_! X}$ is subterminal in $\calS$ then $X$ is subterminal.
\end{lemma}
\begin{proof}
As $p$ is pressential, $p_!$ preserves finite products and finite coproducts. If ${p_! K}$ is initial the isomorphism ${p_! K \rightarrow 0}$ transposes to a map ${K \rightarrow p^* 0 = 0}$, so  $K$ is initial. That is, $p_!$ reflects initial object and so Lemma~\ref{LemExtensiveFolk} is applicable.
\end{proof}

The next result is  well-known. It follows, for example, from A1.5.9 in \cite{elephant}.

\begin{lemma}\label{LemConnectedPressential} 
If ${p: \calE \rightarrow \calS}$ is pressential then, $p$ is connected if and only if ${p^* : \calS \rightarrow \calE}$ is cartesian closed. 
\end{lemma}

If $p$ is essential then, for every $B$ in $\calS$, the geometric morphism ${p/B : \calE/p^* B \rightarrow \calS/B}$ is also essential.
Indeed, if we let ${\tau}$ be the counit of  ${p_! \dashv p^*}$ then the functor ${(p/B)_! : \calE/p^* B \rightarrow  \calS/B}$  sends
each object ${x : X \rightarrow  p^* B}$ in ${\calE/p^* B}$ to the composite
$$\xymatrix{
p_! X \ar[r]^-{p_! x} & p_!(p^* B) \ar[r]^-{\tau} & B
}$$
as an object in ${\calS/B}$. See, for example, Lemma~{5.2} in \cite{LawvereMenni2015}.

Let us say that ${p : \calE \rightarrow \calS}$ is {\em stably pressential} if, for every $B$ in $\calS$, ${p/B : \calE/p^* B \rightarrow \calS/B}$ is pressential. Of course, a stably pressential geometric morphism is pressential.

\begin{proposition}\label{PropDecidableImpliesSubDiscrete} If ${p : \calE \rightarrow \calS}$ is stably pressential then, for every decidable $X$ in $\calE$, the unit ${ X \rightarrow p^* (p_! X)}$ of ${p_! \dashv p^*}$ is monic.
\end{proposition}
\begin{proof}
Assume that $X$ is decidable and let ${B = p_! X}$. 
By Theorem~{11(10)} in \cite{Carboni96}, every map with decidable domain is decidable, so the unit ${\sigma : X \rightarrow p^*(p_! X) = p^* B}$ of ${p_! \dashv p^*}$ is decidable as an object in ${\calE/p^* B}$. By the explicit description of ${(p/B)_! : \calE/p^* B \rightarrow \calS/ B}$ after Lemma~\ref{LemConnectedPressential}, 
${(p/B)_! \sigma = id_{B}}$. That is, ${(p/B)_! \sigma}$ is terminal in $\calS/B$.
Lemma~\ref{LemDecidableSubconnectedImpliesSubterminal} implies that the object ${\sigma : X \rightarrow p^*(p_! X) = p^* B}$ in ${\calE/p^* B}$ is subterminal. That is, the unit  ${\sigma : X \rightarrow p^*(p_! X)}$ is monic in $\calE$.
\end{proof}

If ${p : \calE \rightarrow \calS}$ is hyperconnected then we say that an object $X$ in $\calE$ is {\em discrete} if the counit ${\beta_X : p^* (p_* X) \rightarrow X}$ is an isomorphism. It is well-known that, subobjects of discrete objects are discrete.

\begin{corollary}\label{CorDecidableImpliesDiscrete} If ${p : \calE \rightarrow \calS}$ is stably pressential and hyperconnected then every decidable object in $\calE$ is discrete.
\end{corollary}
\begin{proof}
Every object of the form ${p^* B}$ is discrete and, by Proposition~\ref{PropDecidableImpliesSubDiscrete}, we have a monomorphism ${X \rightarrow p^* B}$ if $X$ is decidable.
\end{proof}

If we strengthen the hypotheses we obtain a characterization.

\begin{corollary}\label{CorDecidableEqualsDiscrete} If $\calS$ is Boolean and ${p : \calE \rightarrow \calS}$ is stably pressential and hyperconnected then, an object in $\calE$ is decidable if and only if it is discrete.
\end{corollary}
\begin{proof}
If $X$ is discrete then it is trivially decidable because $\calS$ is Boolean and ${p^* :\calS \rightarrow \calE}$ preserves finite products and finite coproducts.
The converse follows from Corollary~\ref{CorDecidableImpliesDiscrete}.
\end{proof}

A geometric morphism ${p : \calE \rightarrow \calS}$ is called {\em pre-cohesive} if it is local, hyperconnected and pressential (see \cite{MenniExercise, LawvereMenni2015}). In other words, it is a string of adjoints ${p_! \dashv p^* \dashv p_* \dashv p^!}$ with full and faithful ${p^*, p^! : \calS \rightarrow \calE}$, such that ${p^* : \calS \rightarrow \calE}$ is closed under subobjects and ${p_! :\calE \rightarrow \calS}$ preserves finite products.

We know from \cite{LawvereMenni2015} that, if ${p : \calE \rightarrow  \calS}$ is pre-cohesive then, for every object $B$ in $\calS$, the sliced geometric morphism ${p/B : \calE/p^* B \rightarrow \calS/B}$ is `almost' pre-cohesive, in the sense that all the defining conditions hold except, perhaps,  finite-product preservation of the leftmost adjoint ${(p/B)_! : \calE/p^* B \rightarrow \calS/B}$. For this reason we say that ${p : \calE \rightarrow \calS}$ is {\em stably pre-cohesive} if ${p/B : \calE/p^* B \rightarrow \calS/B}$ is pre-cohesive for every $B$ in $\calS$; that is, if the leftmost adjoints ${(p/B)_! : \calE/p^* B \rightarrow \calS/B}$ preserve finite products.
Alternatively, a pre-cohesive geometric morphism is stably so if and only if it is stably pressential.

\begin{corollary}\label{CorUIAO} If $\calS$ is Boolean and  ${p : \calE \rightarrow \calS}$ 
is a stably pre-cohesive geometric morphism then
${p_* : \calE \rightarrow \calS}$ is a Unity and Identity for the subcategories ${\Un(\calE) \rightarrow \calE}$ and ${\calE_{\neg\neg} \rightarrow \calE}$, making them Adjointly Opposite.
\end{corollary}
\begin{proof}
Since stably pre-cohesive implies stably pressential and hyperconnected, 
Corollary~\ref{CorDecidableEqualsDiscrete} implies that ${p^* : \calS \rightarrow \calE}$ coincides with ${\Un(\calE) \rightarrow \calE}$. By Proposition~{4.4} in \cite{LawvereMenni2015} the subtopos ${p_* \dashv p^! : \calS \rightarrow \calE}$ coincides with  ${\calE_{\neg\neg} \rightarrow \calE}$. 
\end{proof}

We end this section with a brief discussion on the relation between (stably) pre-cohesive, (stably) essential and (stably) locally connected geometric morphisms.

\begin{lemma}\label{LemDefStablyLocc}
If ${p : \calE \rightarrow \calS}$ is a geometric morphism then the following are equivalent:
\begin{enumerate}
\item $p$ is locally connected and ${p_! : \calE \rightarrow \calS}$ preserves finite products,
\item $p$ is connected essential and the left-most adjoint ${p_! : \calE \rightarrow \calS}$ sends pullbacks
$$\xymatrix{
P \ar[d]_-{\pi_0} \ar[r]^-{\pi_1} & Y \ar[d]^-y \\
X \ar[r]_-x & p^* B
}$$
in $\calE$ to pullbacks in $\calS$ (in other words, the adjunction ${p_! \dashv p^*}$ has stable units),
\item $p$ is connected essential and ${(p/B)_! : \calE/p^* B \rightarrow \calS/B}$ preserves finite products for every $B$ in $\calS$.
\item $p$ is connected and stably pressential.
\end{enumerate}
\end{lemma}
\begin{proof}
The equivalence of the first three items is proved in Proposition~{10.2} in \cite{LawvereMenni2015}. We have added the fourth item for emphasis which is, almost by definition, equivalent to the third.
\end{proof}

Geometric morphisms $p$ satisfying the conditions in the first item of Lemma~\ref{LemDefStablyLocc} are called {\em stably locally connected} by Johnstone in \cite{Johnstone2011}. (These are always connected.) As mentioned in \cite{LawvereMenni2015}, he informed us that the terminology was chosen by analogy with `stably locally compact'. On the other hand, the word ``stably" in ``stably pressential" refers to stability under slicing of the property that the leftmost adjoint preserves finite products.  In parallel, notice that the notion of local connectedness is also a form of stability in the latter sense, namely, stability under slicing of the property that inverse image is cartesian closed. 
In the arguments we use, stability of finite-product preservation of the leftmost adjoint seems more immediately applicable; for example, as in Lemma~\ref{LemDecidableSubconnectedImpliesSubterminal}.

As a corollary of Lemma~\ref{LemDefStablyLocc} one gets that a pre-cohesive geometric morphism is stably so if and only if it is locally connected (see Corollary~{10.4} in \cite{LawvereMenni2015}). 
We still do not know if every pre-cohesive geometric morphism is stably so.

\section{A remark about the Nullstellensatz}

Let ${p : \calE \rightarrow \calS}$ be a geometric morphism. 
If $p$ is connected and essential then there is a canonical natural transformation ${\theta : p_* \rightarrow p_!}$ and, following \cite{Lawvere07}, we say that $p$ satisfies the {\em Nullstellensatz} if ${\theta_X : p_* X \rightarrow p_! X}$ is epic for every $X$ in $\calE$.

It follows from \cite{Johnstone2011} that, if $p$ is essential and local (and hence connected), $p$ satisfies the Nullstellesatz if and only if $p$ is hyperconnected. See also \cite{LawvereMenni2015}. In this case, it follows from Lemma~{4.1} in \cite{LawvereMenni2015} that the rightmost adjoint ${p^! : \calS \rightarrow \calE}$ preserves $0$. (We have already used this tacitly via the invocation to Proposition~{4.4} loc.~cit. in the proof of Corollary~\ref{CorUIAO} above.)

F.~Marmolejo once pointed my attention to the fact that Lemma~{4.1} in \cite{LawvereMenni2015} (saying that the codiscrete inclusion ${p^! : \calS \rightarrow \calE}$ is dense) could be seen as a sufficient condition for a direct image to reflect $0$. This observation leads to the following  variant of that result.

\begin{lemma}\label{LemMarmolejo} If ${p : \calE \rightarrow \calS}$ is connected essential, and the Nullstellensatz holds then each of the items below implies the next one
\begin{enumerate}
\item ${p_* X}$ is initial,
\item ${p_! X}$ is initial,
\item $X$ is initial,
\end{enumerate}
for every $X$ in $\calE$.
If, moreover, $p_*$ preserves $0$ then the three items are equivalent.
\end{lemma}
\begin{proof}
The first item implies the second because ${p_* X \rightarrow p_! X}$ is epic by the Nullstellensatz. The second item implies the third because we can transpose ${p_! X \rightarrow 0}$ to ${X \rightarrow p^* 0 = 0}$. 
\end{proof}

Intuitively, a space is empty iff it has no points iff it has no pieces.

It follows that if ${p : \calE \rightarrow \calS}$ is local, essential and satisfies the Nullstellensatz then ${p^! 0 = 0}$.
Indeed, since $p^!$ is fully faithful, ${p_* (p^! 0) = 0}$, so ${p^! 0}$ is initial by Lemma~\ref{LemMarmolejo}.
Compare with the proof of Lemma~{4.1} in \cite{LawvereMenni2015}.
We give below a strengthening of this result, but first we need a couple of remarks about hyperconnected geometric morphisms.

\begin{lemma}\label{LemForNewDenseness} If the geometric morphism ${p : \calE \rightarrow \calS}$ is hyperconnected then ${p_* : \calE \rightarrow \calS}$ is faithful on morphisms whose domain is discrete.
\end{lemma}
\begin{proof}
Let ${f, g : p^* A \rightarrow X}$ be morphisms  in $\calE$ and assume that ${p_* f = p_* g : p_* (p^* A) \rightarrow p_* X}$. Then the following diagram commutes
$$\xymatrix{
p^* (p_* (p^* A)) \ar[d]_-{\beta} \ar[rr]^-{p^* (p_* f)}_-{p^* (p_* g)} & & p^*(p_* X) \ar[d]^-{\beta} \\
p^* A \ar[rr]<+1ex>^-f \ar[rr]<-1ex>_-g & & X
}$$
and, as the left vertical map is epic, ${f = g}$.
\end{proof}

For the next remark it is convenient to distinguish notationally the subobject classifiers of $\calE$ and $\calS$. Let ${\tau : p_* (\Omega_{\calE}) \rightarrow \Omega_{\calS}}$ be the classifying map of ${p_* \top : p_* 1 \rightarrow p_* (\Omega_{\calE})}$.  Proposition~{A4.6.6} in \cite{elephant} implies that  $p$ is hyperconnected if and only if  $\tau$ is an isomorphism. For many arguments we will not need this notational distinction; we can just use that ${p_* \top : p_* 1 \rightarrow p_* \Omega}$ is a subobject classifier in $\calS$.

\begin{proposition}\label{PropNewDenseness} Let ${p : \calE \rightarrow \calS}$ be a local geometric morphism. If $p$ is hyperconnected then ${p^! : \calS \rightarrow \calE}$ preserves $0$. If $\calS$ is Boolean then the converse holds.
\end{proposition}
\begin{proof}
First observe that if $p$ is local then ${p^! 0}$ is subterminal. Indeed, for any $X$ in $\calE$, maps ${X \rightarrow p^! 0}$ are in bijective correspondence with maps ${p_* X \rightarrow 0}$ so there can be at most one.

Let ${\chi : 1 \rightarrow \Omega}$ be the unique map in $\calE$ such that the diagram on the left below 
$$\xymatrix{
 p^! 0 \ar[d] \ar[r] & 1 \ar[d]^-{\top} && 
0 =  p_* (p^! 0) \ar[d] \ar[r] & p_* 1 \ar[d]^-{p_* \top}  \\
1 \ar[r]_-{\chi} & \Omega &&
  p_* 1 \ar[r]_-{p_* \chi} & p_* \Omega
}$$
is a pullback. As ${p_*}$ preserves pullbacks, the diagram on the right above is also a pullback (in $\calS$). Similarly, the square below
$$\xymatrix{
0 = p_* 0 \ar[d] \ar[r] & p_* 1 \ar[d]^-{p_* \top} \\
p_* 1 \ar[r]_-{p_* \bot} & p_* \Omega
}$$
is a pullback in $\calS$. As ${p_* \Omega}$ is a subobject classifier of $\calS$, we can deduce that ${p_* \chi = p_* \bot :  p_* 1 \rightarrow p_* \Omega}$. Lemma~\ref{LemForNewDenseness} implies that ${\chi = \bot : 1 \rightarrow \Omega}$ and so, ${p^! 0}$ is initial.

Finally, consider the statement that for  $\calS$  Boolean and ${p : \calE \rightarrow \calS}$ local, ${p^! 0 = 0}$ implies $p$ hyperconnected. This is just Lemma~{4.2} in \cite{LawvereMenni2015} which, although stated differently, proves exactly this.
\end{proof}

Let us stress that, as witnessed by some Grothendieck toposes of monoid actions, $p$ hyperconnected does not imply that $p_*$ reflects initial object.
For a concrete example, consider the topos of actions of the additive monoid of natural numbers. 

Altogether,  with the Nullstellensatz in mind, one is led to the consideration of hyperconnected geometric morphisms ${p: \calE \rightarrow \calS}$ such that $p_*$ reflects $0$.

\section{Hyperconnected morphisms with Boolean codomain}

The purpose of this section is to show that if ${\Un(\calE) \rightarrow \calE}$ has a right adjoint ${p_* : \calE \rightarrow \Un(\calE)}$ as in Lemma~\ref{LemHyperConnectedIffQuotients} then, this right adjoint is a UI for the subcategories ${\Un(\calE) \rightarrow \calE}$ and ${\calE_{\neg\neg}\rightarrow  \calE}$ if and only if  $p_*$ reflects initial object.

(For convenience we state the following result with the notation for geometric morphisms that we use in the application, but the reader will immediately notice that it is a general simple fact about adjunctions.)

\begin{lemma}\label{LemPreambleToConnectedImpliesLocal}
Let ${p^* : \calS \rightarrow \calE}$ be a coreflective subcategory with right adjoint $p_*$ and  counit ${\beta : p^* p_* \rightarrow Id_{\calE}}$,   and let  
${f^* \dashv f_* : \calF \rightarrow \calE}$ be a reflective subcategory with unit ${\eta : Id_{\calE} \rightarrow f_* f^*}$. 
If the natural transformations ${f^* \beta_{f_*} : f^* p^* p_* f_* \rightarrow f^* f_*}$ and ${p_* \eta_{p^*} : p_* p^* \rightarrow p_* f_* f^* p^*}$ are isos then  the composite adjunction ${f^* p^* \dashv p_* f_* : \calF \rightarrow \calS}$ is an adjoint equivalence.
\end{lemma}
\begin{proof}
Let ${\epsilon}$ be the counit of ${f^* \dashv f_*}$ and ${\alpha}$ be the unit of ${p^* \dashv p_*}$. Both $\epsilon$ and $\alpha$ are isomorphisms by hypothesis.
It is well-known that the composite adjunction ${f^* p^* \dashv p_* f_* : \calF \rightarrow \calS}$ has the following unit and counit
$$\xymatrix{
1_{\calS} \ar[r]^-{\alpha} & p_* p^* \ar[r]^-{p_* \eta_{p^*}} & p_*f_* f^* p^* &
f^* p^* p_* f_* \ar[r]^-{f^* \beta_{f_*}} & f^* f_* \ar[r]^-{\epsilon} & 1_{\calF}
}$$
and, clearly, the four maps above are isomorphisms by hypothesis. Therefore, both the unit and counit of the composite adjunction are isomorphisms. That is, ${f^* p^* \dashv p_* f_* : \calF \rightarrow \calS}$ is an adjoint equivalence.
\end{proof}

The following is also a simple general fact about adjunctions but let us formulate it in terms of toposes.

\begin{lemma}\label{LemConnectedImpliesLocal} Let ${p : \calE \rightarrow \calS}$ be a connected geometric  morphism with counit $\beta$ and let ${f : \calF \rightarrow \calE}$ be a subtopos with unit $\eta$. If the natural ${f^* \beta : f^* p^* p_* \rightarrow f^*}$ and ${p_* \eta_{p^*} : p_* p^* \rightarrow p_* f_* f^* p^*}$ are isomorphisms then $p_*$ has a right adjoint ${p^!}$ and the subtopos ${p_* \dashv p^! : \calS \rightarrow \calE}$ coincides with $f$. (In other words, $p$ is local and its center coincides with $f$.)
\end{lemma}
\begin{proof}
The hypothesis that ${f^* \beta : f^* p^* p_* \rightarrow f^*}$ is an iso means that the following diagram
$$\xymatrix{
 & \ar@(l,u)[ld]_-{p_*} \calE \ar@(r,u)[rd]^-{f^*} & \\ 
\calS  \ar[r]_-{p^*} & \calE \ar[r]_-{f^*}  & \calF
}$$
commutes up to (that) canonical isomorphism. The same hypothesis trivially implies  that ${f^* \beta_{f_*} : f^* p^* p_* f_* \rightarrow f^* f_*}$ is an iso. Then Lemma~\ref{LemPreambleToConnectedImpliesLocal} is applicable and so the bottom composite in the diagram above is an equivalence. Therefore, the composite 
$$\xymatrix{
\calS \ar[r]^-{p^*} & \calE \ar[r]^-{f^*} & \calF \ar[r]^-{f_*} & \calE
}$$
is a right adjoint $p^!$ to ${p_* : \calE \rightarrow \calS}$. Clearly,  ${p^! : \calS \rightarrow \calE}$ and ${f_* : \calF \rightarrow \calE}$ are equivalent over $\calE$. 
\end{proof}

From now on let ${p : \calE \rightarrow \calS}$ be a geometric morphism.
Recall that ${\tau : p_* \Omega_{\calE}  \rightarrow \Omega_{\calS}}$ is the unique map such that the following square
$$\xymatrix{
p_* \top \ar[d]_-{p_* \top} \ar[r]^-{!} & 1 \ar[d]^-{\top}  \\
p_* \Omega_{\calE} \ar[r]_-{\tau} & \Omega_{\calS}
}$$
is a pullback in $\calS$.  Since both ${p_* \Omega_{\calS}}$ and ${\Omega_{\calS}}$ are canonically equipped with a Heyting algebra structure it is natural to ask how much of that structure is preserved by $\tau$. We will not address the full question here, but only the fragment we need which, incidentally, is probably known, although we have not found it in the literature.

\begin{lemma}\label{LemHyperconDirectImagePreservesNegation} If ${p_* : \calE \rightarrow \calS}$ preserves $0$ then the diagram on the left below commutes in $\calS$
$$\xymatrix{
p_* 1 \ar[d]_-{p_* \bot} \ar[r]^-{!} & 1 \ar[d]^-{\bot} &&   
   p_* \Omega_{\calE} \ar[d]_-{p_* \neg} \ar[r]^-{\tau} & \Omega_{\calS} \ar[d]^-{\neg} \\
p_* \Omega_{\calE} \ar[r]_-{\tau} & \Omega_{\calS} && 
   p_* \Omega_{\calE} \ar[r]_-{\tau} & \Omega_{\calS}
}$$
so,  if ${p : \calE \rightarrow \calS}$ is hyperconnected then the diagram on the right above commutes.
Hence, in this case, ${p_* : \calE \rightarrow \calS}$ preserves Heyting complements of subobjects.
\end{lemma}
\begin{proof}
In the following  diagram
$$\xymatrix{
0 \ar[d]_-{!} \ar[r]^-{!} & p_* 0 \ar[d]_-{p_* !} \ar[r]^-{p_* !} & p_* 1 \ar[d]_-{p_* \top} \ar[r]^-{!} & 1 \ar[d]^-{\top} \\
1 \ar[r]_-{!} & p_* 1 \ar[r]_-{p_* \bot} & p_* \Omega_{\calE} \ar[r]_-{\tau} & \Omega_{\calS}
}$$
all squares are pullbacks. Indeed, the right square is a pullback by definition of $\tau$, the middle one is so because $p_*$ preserves pullbacks, and the left one is a pullback because $p_*$ preserves $0$ by hypothesis.
It follows that the bottom composite equals ${\bot : 1 \rightarrow \Omega_{\calS}}$, which means that the left square in the statement commutes.

If $p$ is hyperconnected then ${p_*}$ preserves $0$ because it is a coreflection. Moreover, in this case, by Proposition~{A4.6.6} in \cite{elephant}, $\tau$ is an isomorphism; so the square that we have just proved commutative is actually a pullback. Therefore, all the squares in the diagrams below are pullbacks
$$\xymatrix{
p_* 1 \ar[d]_-{p_* \bot} \ar[r]^-{!} & 1 \ar[d]^-{\bot}  \ar[r]^-{!} & 1 \ar[d]^-{\top}    &&
   p_* 1 \ar[d]_-{p_* \bot} \ar[r]^-{!} & p_* 1 \ar[d]^-{p_* \top}  \ar[r]^-{!} & 1 \ar[d]^-{\top}   \\
p_* \Omega_{\calE} \ar[r]_-{\tau} & \Omega_{\calS} \ar[r]_-{\neg} &  \Omega_{\calS} &&
p_* \Omega_{\calE} \ar[r]_-{p_* \neg} &  p_* \Omega_{\calE} \ar[r]_-{\tau} & \Omega_{\calS}
}$$
so the bottom composites coincide; that is, the right square in the statement commutes. It is now easy to check that ${p_* : \calE \rightarrow \calS}$ preserves Heyting complements.
First observe that if the square on the left below is a pullback in $\calE$ 
$$\xymatrix{
U \ar[d]_-u \ar[r]^-{!} & 1 \ar[d]^-{\top} && 
   p_* U \ar[d]_-{p_* u} \ar[r]^-{!} & p_* 1 \ar[d]^-{p_* \top} \ar[r]^-{!} & 1 \ar[d]^-{\top} \\
X \ar[r]_-{\chi_u} & \Omega_{\calE} && 
  p_* X \ar[r]_-{p_* \chi_u} & p_* \Omega_{\calE} \ar[r]_-{\tau}  &  \Omega_{\calS} 
}$$
then the rectangle on the right above is also a pullback (in $\calS$), so its bottom composite must be the classifying morphism of the subobject ${p_* u : p_* U \rightarrow p_* X}$. In other words, ${\tau (p_* \chi_u) = \chi_{p_* u} : p_* X \rightarrow \Omega_{\calS}}$.

Finally, the Heyting complement ${\neg u : \neg U \rightarrow X}$ of ${u : U \rightarrow X}$ is classified by the composite on the left below
$$\xymatrix{
X \ar[r]^-{\chi_u} & \Omega_{\calE} \ar[r]^-{\neg} & \Omega_{\calE} & 
  p_* (\neg U) \ar[d]_-{p_* (\neg u)} \ar[rr]  & & p_* 1 \ar[d]^-{p_*\top} \ar[r]^-{!} & 1 \ar[d]^-{\top} \\
 & &  &
  p_* X \ar[r]_-{p_* \chi_u} & p_* \Omega_{\calE} \ar[r]_-{p_* \neg} & p_* \Omega_{\calE} \ar[r]_-{\tau} & \Omega_{\calS} 
}$$
so the pullback diagram on the right above shows that its bottom composite classifies ${p_* (\neg u) : p_*(\neg U) \rightarrow p_* X}$. Since the diagram below commutes
$$\xymatrix{
  p_* X \ar[rd]_-{\chi_{p_* u}} \ar[r]^-{p_* \chi_u} & p_* \Omega_{\calE} \ar[d]_-{\tau} \ar[r]^-{p_* \neg} & p_* \Omega_{\calE} \ar[d]^-{\tau} \\
& \Omega_{\calS} \ar[r]_-{\neg} & \Omega_{\calS} 
}$$
it follows that ${p_* (\neg u)}$ is the same subobject of ${p_* X}$ as ${\neg (p_* u)}$.
\end{proof}

We will need to apply the next observation twice.

\begin{lemma}\label{LemDense} Assume that ${p_* : \calE \rightarrow \calS}$ both preserves and reflects $0$,  and let   ${u :  U \rightarrow X}$ be a monomorphism. If ${p_* u }$ is an isomorphism then ${\neg U}$ is initial and therefore $u$ is $\neg\neg$-dense.
\end{lemma}
\begin{proof}
 Let ${\neg u : \neg U \rightarrow X}$ denote the Heyting complement of $u$. Then the diagram on the left below is a pullback
$$\xymatrix{
0 \ar[d] \ar[r] &  \neg U \ar[d]^-{\neg u} && 0 \ar[d] \ar[r] &  p_* (\neg U) \ar[d]^-{p_* (\neg u)} \\
U \ar[r]_-{u} & X && p_* U \ar[r]_-{p_* u} & p_* X
}$$
and, as  $p_*$ preserves $0$ by hypothesis, the diagram on the right above is also a pullback. As ${p_* u}$ is an isomorphism, so is ${0 \rightarrow   p_* (\neg U)}$.
Since $p_*$ reflects $0$ by hypothesis, ${ \neg U }$ is initial.
\end{proof}

The following result is a strengthening of that outlined in \S 2 of \cite{McLarty87} (see also \cite{McLarty88}). 
It is also related to Theorem~{3.4} in \cite{Johnstone2011} in the vague sense that a geometric morphism whose domain is a `topos of spaces' is actually local. Indeed, notice the invocation to Lemma~{3.1} loc.~cit. in the proof below.

\begin{proposition}\label{PropSeparatedUnitDenseCounit}
If $\calS$ is a Boolean topos and ${p: \calE \rightarrow \calS}$ is a hyperconnected geometric morphism then, $p$ is local 
 if and only if the following hold:
\begin{enumerate}
\item the functor ${p_* : \calE \rightarrow \calS}$ reflects $0$ and
\item for every $A$ in $\calS$,  ${p^* A}$ is $\neg\neg$-separated.
\end{enumerate}
In this case, the subtopos ${p_* \dashv p^! : \calS \rightarrow \calE}$ coincides with ${\calE_{\neg\neg} \rightarrow \calE}$.
\end{proposition}
\begin{proof}
Assume first that  $p$ is local and hyperconnected. Then ${p^! : \calS \rightarrow \calE}$ preserves $0$ by Proposition~\ref{PropNewDenseness}. By generalities about reflective subcategories and strict initial objects, $p_*$ reflects initial object if and only if $p^!$ preserves it. So it remains to show that ${p^* A}$ is $\neg\neg$-separated.
For this, notice that the subtopos ${p_* \dashv p^! : \calS \rightarrow \calE}$ is a Boolean dense subtopos and so it must coincide with the subtopos of sheaves for the double negation topology. 
Moreover, if we let ${\eta}$ be the unit of ${p_* \dashv p^!}$ then, by Lemma~{3.1} in \cite{Johnstone2011}, $p$ is hyperconnected if and only if  the canonical natural transformation ${p^* \rightarrow p^!}$ is monic; but this is equivalent to  ${\eta_{p^*} : p^* A \rightarrow p^!  (p_* (p^* A))}$ being monic for every $A$ in $\calS$. In turn, this is equivalent to  ${p^* A}$ being separated (w.r.t. ${p_* \dashv p^!}$) for every $A$ in $\calS$ (see Lemma~{A4.3.6} in \cite{elephant}).

For the converse we need to assume that the two items in the statement hold and prove that $p_*$ has a right adjoint. The strategy of the proof is to apply Lemma~\ref{LemConnectedImpliesLocal}, so consider the subtopos ${f : \calE_{\neg\neg} \rightarrow \calE}$ of sheaves for the $\neg\neg$-topology. Let $\beta$ be the (monic) counit of ${p^* \dashv p_*}$ and ${\eta}$ be the unit of ${f^* \dashv f_*}$.

By Lemma~\ref{LemDense}, the counit ${\beta_X : p^* (p_*  X) \rightarrow X}$ is $\neg\neg$-dense for every $X$ in $\calE$ and, hence, ${f^* \beta}$ is an isomorphism.

For the same reasons mentioned in the first paragraph of the proof,  ${p^* A}$ is $\neg\neg$-separated if and only if  ${\eta_{p^*} : p^* A \rightarrow f_* (f^* (p^* A))}$ is monic.  Since ${f^* : \calE \rightarrow \calE_{\neg\neg}}$ reflects $0$ and ${f^* \eta}$ is an isomorphism, Lemma~\ref{LemDense} implies that ${\neg\eta_{p^*} : \neg (p^* A ) \rightarrow f_* (f^* (p^* A))}$ is the initial subobject.
Lemma~\ref{LemHyperconDirectImagePreservesNegation} then implies that 
$${\neg(p_* \eta_{p^*}) = p_* (\neg \eta_{p^*}) = p_* 0 = 0}$$
 as subobjects of ${p_* (f_* (f^* (p^* A)))}$.
Since the topos ${\calS}$ is Boolean, $${p_* \eta_{p^*} : p_* (p^* A)  \rightarrow p_* (f_* (f^* (p^* A)))}$$ is an isomorphism.
So we can apply Lemma~\ref{LemConnectedImpliesLocal} to complete the proof.
\end{proof}

To summarize, we state another sufficient condition for the existence of a Unity and Identity for decidable objects and $\neg\neg$-sheaves. (Again, compare with \cite{McLarty87}.)

\begin{corollary}\label{CorMotivatedByMcLarty} Let $\calE$ be a topos and assume that the inclusion ${\Un(\calE) \rightarrow \calE}$ has a right adjoint ${p_*}$.  Then the resulting  hyperconnected  ${p : \calE \rightarrow \Un(\calE)}$  is local if and only if  $p_*$  reflects 0. 
In this case, the subtopos ${p_* \dashv p^! :\Un(\calE) \rightarrow \calE}$ coincides with ${\calE_{\neg\neg} \rightarrow \calE}$. 
Also in this case, $p$ is pre-cohesive if and only if ${p^* : \Un(\calE) \rightarrow \calE}$ is cartesian closed.
\end{corollary}
\begin{proof}
By  Lemma~\ref{LemHyperConnectedIffQuotients}, the existence of $p_*$ determines the hyperconnected $p$. Proposition~\ref{PropSeparatedUnitDenseCounit}, together with the well-known fact that decidable objects are $\neg\neg$-separated, implies that $p$ is local if and only if $p_*$ reflects $0$. By Corollary~{3.11} in \cite{Menni2017a}, the local $p$ is pre-cohesive if and only if $p^*$ is cartesian closed.
\end{proof}

In analogy with  \cite{LawvereMenni2015}, one interpretation of the results above involves the conclusion that the real contrast between {\em Mengen} and {\em Kardinalen} emerges from the case of a topos $\calE$ whose subcategory ${\Un(\calE) \rightarrow \calE}$ of decidable objects has additional remarkable properties. In particular, that it has a $0$-reflecting right adjoint.

\section*{Acknowledgments} I thank F.~W.~Lawvere and F.~Marmolejo for several useful discussions. I also thank C.~McLarty who sent me copies of his papers \cite{McLarty87, McLarty88}. Much of the work was done during a visit to the {\em Universit\`a di Bologna} in 2017, with the support of C.~Smith and funding from the European Union's Horizon 2020 research and innovation programme under the Marie Sk\l odowska-Curie grant agreement No.~{690974}.


\end{document}